\documentclass{article}[12pts]

\usepackage{amssymb}
\usepackage{a4}
\usepackage{QED}

\def\f{\rightarrow}
\def\q{\forall}
\def\non{\neg}
\def\G{\Gamma}
\def\D{\Delta}
\def\v{\vdash}
\def\ou{\vee}
\def\et{\wedge}

\def\F{\displaystyle\frac}

\def\fl{\ifmmode{-\mkern-3.5mu
\circ\mkern2mu}\else{-\kern-.28em 
\circ\hskip1pt\ }\fi\relax} 

\newtheorem{definition}{Definition}[section]
\newtheorem{theorem}{Theorem}[section]
\newtheorem{lemma}{Lemma}[section]
\newtheorem{remark}{Remark}[section]
\newtheorem{corollary}{Corollary}[section]

\begin{document} 

\begin{center} 
{\Large\bf PARAMETRIC MIXED SEQUENT CALCULUS}\\[1cm]

\begin{tabular}{cc}
{\large \bf Olivier LAURENT} & {\large \bf Karim NOUR}\\[0.2cm]
Universit\'e Denis Diderot & Universit\'e de Savoie\\
PPS - Case 7014 & LAMA - Equipe de logique \\
2, place Jussieu & Campus scientifique\\
75251 Paris Cedex 05 - France & 73376 Le Bourget du Lac - France  \\
\end{tabular}
\end{center}

\begin{abstract}
In this paper, we present a propositional sequent calculus containing
disjoint copies of classical and intuitionistic logics. We prove a
cut-elimination theorem and we establish a relation between this
system and linear logic.
\end{abstract}

\section{Introduction}

The systems which represent classical and intuitionistic logics have
often been studied separately. For each of these logics we establish
its proper properties. Although the systems which represent
intuitionistic logic are sub-systems of those which represent
classical logic, we usually do not find in the literature studies of the
properties of these systems at the same time. We mainly find many
external translations from classical to intuitionistic logic. The main
novelty of the system which we look for is that intuitionistic and
classical logics appear as fragments. For instance a proof of an
intuitionistic formula may use classical lemmas without any
restriction. This approach is radically different from the one that
consists in changing the rule of the game when we want to change the
logic. We want only one logic which, depending on its use, may appear
classical or intuitionistic.\\

J.-Y. Girard presented in \cite{LU} a single sequent calculus (denoted
LU) common to classical, intuitionistic and linear logics.  Each
formula is given with a polarity: positive, neutral and negative. For
each connective the rules depend on the polarity of the formulas. On
the other hand the system LU has a cut-elimination theorem and then
the sub-formula property. Although the system LU is an answer to our
question, we seek a simpler and more intuitive system only for
classical and intuitionistic logics.\\

The usefulness of finding a system which mixes classical and
intuitionistic logics is more and more recognized. For example,
J.-L. Krivine and the second author introduced a second order mixed
logic in order to type storage and control operators in $\lambda {\cal
C}$-calculus (see \cite{KN}). Indeed, they needed intuitionistic logic
to characterize the operational behavior of the storage operators and
classical logic for the control operators. The theoretical properties
of this system are not difficult to prove because the only connectives
are $\f$ and $\q$. Recently, C. Raffalli introduced in \cite{Ra} a
second order mixed logic which is slight extension of that of \cite{KN}
in order to extract a program for some classical proofs. He applied
his method to extract a program from a classical proof of Dickson's
lemma.\\

In the paper \cite{NN}, the second author and A. Nour presented a
propositional logic with all connectives (denoted PML) containing
three kinds of variables: minimal, intuitionistic and classical. The
absurdity rules are restricted to the formulas containing the
corresponding variables. They introduced for the system PML a Kripke
semantics and they showed a completeness theorem. They deduced from this
theorem a very significant result which is the following: ``for a
formula to be derivable in a logic, it is necessary that the formula
contains at least a variable which corresponds to this logic''. They
also presented a sequent calculus version of this system. The systems
presented in \cite{NN} are not satisfactory because they do not have
cut-elimination results. \\

We propose in this paper another approach to solve this problem. We
fix a set of formulas ${\cal P}$ which represents intuitively the set
of formulas on which we can do classical reasoning.  A sequent is a
pair of multisets of formulas, denoted $\G \v \D;\Pi$, where $\D
\subseteq {\cal P}$ and $\Pi$ contains at most one formula.  We
introduce a list of classical rules on the multiset $\D$ and
intuitionistic rules on $\Pi$. Certain rules will require conditions
on the membership of some formulas to the set ${\cal P}$. We 
prove a cut-elimination
theorem and thus deduce the sub-formula property.
We show how to code systems LK and LJ in our system which is
coded in system LL.\\

The paper is organized as follows. In section \ref{Def} we present the
rules of our system.  We prove in section \ref{Prop} the cut
elimination properties of the system. The codings of systems LK and LJ
in our system are given in section \ref{Codings}. We present in section
\ref{LL} a coding of the system in system LL.

\section{The ML$_{\cal P}$ sequent calculus}\label{Def}

\begin{definition} $\;$
\begin{enumerate}
\item The set of formulas is defined by the following grammar:
$$
F ::= \; 0  \; \mid  \; \perp  \; \mid  \; X  \; \mid  \; F \et F  \;
\mid  \; F \ou F  \; \mid  \; F \f F
$$
where $X$ ranges over a set of propositional variables ${\cal V}$.
\item Let ${\cal P}$ be a subset of formulas. A \emph{${\cal P}$-sequent} is a pair of multisets
of formulas, denoted $\G \v \D;\Pi$, where $\D \subseteq {\cal P}$ and
$\Pi$ contains at most one formula. The set $\D$ is called the \emph{body}
and $\Pi$ (the space after ``;'') is called the \emph{stoup}. A \emph{${\cal
P}$-derivation} may be constructed according to one of the rules below.

\begin{center}
{\bf \large AXIOM/CUTS}
\end{center}

\begin{center}
$\F{}{A \v ;A} \; ax$
\end{center}

\begin{minipage}[t]{200pt}
$\F{\G \v \D;A \;\;\; \G',A \v \D';\Pi} {\G, \G' \v \D,\D';\Pi} \; cut_1$
\end{minipage}
\begin{minipage}[t]{200pt}
$\F{\G \v \D,A;\Pi \;\;\; \G',A \v \D';} {\G, \G' \v \D,\D';\Pi} \; cut_2$
\end{minipage}

\begin{center}
{\bf \large STRUCTURE}
\end{center}

\begin{minipage}[t]{200pt}
$\;$ 
\end{minipage}
\begin{minipage}[t]{200pt}
$\F{\G \v \D;A  \; {\;}_{A \in {\cal P}}}{\G \v \D,A;} \; der$ 
\end{minipage}

\begin{minipage}[t]{200pt}
$\F{\G,A,A \v \D;\Pi}{\G,A \v \D;\Pi} \; c_l$ 
\end{minipage}
\begin{minipage}[t]{200pt}
$\F{\G \v \D,A,A;\Pi}{\G \v \D,A;\Pi} \; c_r$ 
\end{minipage}

\begin{minipage}[t]{200pt}
$\F{\G\v \D;\Pi}{\G,A \v \D;\Pi} \; w_l$
\end{minipage}
\begin{minipage}[t]{200pt}
$\F{\G \v \D;\Pi \;  {\;}_{A\in{\cal P}}}{\G \v \D,A;\Pi} \; w_r$
\end{minipage}

\begin{center}
{\bf \large LOGIC}
\end{center}

\begin{minipage}[t]{200pt}
$\F{{\;}_{\D \subseteq {\cal P}}}{\G,0 \v \D;\Pi} \; 0$
\end{minipage}
\begin{minipage}[t]{200pt}
$\;$\\[0.2cm]
\end{minipage}
\begin{minipage}[t]{200pt}
$\F{}{\perp \v ;} \; \perp$
\end{minipage}
\begin{minipage}[t]{200pt}
$\;$\\
\end{minipage}

\begin{minipage}[t]{200pt}
$\F{\G,A,B \v \D;C \; {\;}_{A \not \in {\cal
    P} \; {\rm and} \; B \not \in {\cal P}}} {\G, A \et B \v \D;C} \; \et^1_l$
\end{minipage}
\begin{minipage}[t]{200pt}
$\F{\G \v \D;A \;\;\; \G'\v\D';B} {\G,\G' \v \D,\D'; A \et B} \; \et^1_r$
\end{minipage}

\begin{minipage}[t]{200pt}
$\F{\G,A,B \v \D;} {\G, A \et B \v \D;} \; \et^2_l$
\end{minipage}
\begin{minipage}[t]{200pt}
$\F{\G \v \D,A; \;\;\; \G'\v\D',B;} {\G,\G' \v \D,\D'; A \et B} \; \et^2_r$
\end{minipage}

\begin{minipage}[t]{200pt}
$\;$
\end{minipage}
\begin{minipage}[t]{200pt}
$\F{\G \v \D; A\;\;\; \G'\v\D',B;} {\G,\G' \v \D,\D'; A \et B} \; \et^3_r$ 
\end{minipage}

\begin{minipage}[t]{200pt}
$\;$
\end{minipage}
\begin{minipage}[t]{200pt}
$\F{\G \v \D,A; \;\;\; \G'\v\D';B} {\G,\G' \v \D,\D'; A \et B} \; \et^4_r$ \\
\end{minipage}

\begin{minipage}[t]{200pt}
$\F{\G , A \v \D ;C  \;\;\;  \G , B \v \D ; C \; {\;}_{A \not \in {\cal
    P} \; {\rm and} \; B \not \in {\cal P}}} 
{\G , A \ou B \v \D ; C} \; \ou^1_l$
\end{minipage}
\begin{minipage}[t]{200pt}
$\F{\G \v \D;A} { \G \v \D ; A \ou B} \; \ou^1_r$ 
\end{minipage}

\begin{minipage}[t]{200pt}
$\;\;\;\;\;\;\;\;\;\;\;\;\;\;\;\;\;\;\;\;\;\;\;\;\;$
\end{minipage}
\begin{minipage}[t]{200pt}
$\F{\G \v \D;B} { \G \v \D ; A \ou B} \; \ou^2_r$ 
\end{minipage}

\begin{minipage}[t]{200pt}
$\F{\G , A \v \D ;  \;\;\;  \G , B \v \D ;} 
{ \G , A \ou B \v \D ;} \; \ou^2_l$
\end{minipage}
\begin{minipage}[t]{200pt}
$\F{\G \v \D,A;} { \G \v \D ; A \ou B} \; \ou^3_r$ 
\end{minipage}

\begin{minipage}[t]{200pt}
$\;$
\end{minipage}
\begin{minipage}[t]{200pt}
$\F{\G \v \D,B;} { \G \v \D ; A \ou B} \; \ou^4_r$ \\
\end{minipage}

\begin{minipage}[t]{200pt}
$\F{\G,B \v \D;C  \;\;\; \G'\v\D';A  \; {\;}_{B \not \in {\cal P}}}{\G,\G', A \f B \v
    \D,\D';C} \; \f^1_l$
\end{minipage}
\begin{minipage}[t]{200pt}
$\F{\G , A \v  \D;B} {\G \v \D ; A \f B } \; \f^1_r$
\end{minipage}

\begin{minipage}[t]{200pt}
$\F{\G,B \v \D;  \;\;\; \G'\v\D';A}{\G,\G', A \f B \v \D,\D';} \; \f^2_l$
\end{minipage}
\begin{minipage}[t]{200pt}
$\F{\G , A \v  \D,B;} {\G \v \D ; A \f B} \; \f^2_r$
\end{minipage}

\begin{minipage}[t]{200pt}
$\F{\G,B \v \D;  \;\;\; \G'\v\D',A;\Pi}{\G,\G', A \f B \v \D,\D';\Pi} \; \f^3_l$
\end{minipage}
\begin{minipage}[t]{200pt}
$\;\;\;\;\;\;\;\;\;\;\;\;\;\;\;\;\;\;\;\;\;\;\;\;\;$
\end{minipage}
\end{enumerate}
\end{definition}

We write $\G  \v_{\cal P} \D; \Pi$ if the ${\cal P}$-sequent $\G  \v \D;
\Pi$ is derivable in system ML$_{\cal P}$.

\begin{remark}
\begin{enumerate}
\item The conditions which we add on the set ${\cal P}$ in some left
logical rules are necessary to obtain a cut-elimination
theorem. Indeed without these conditions the cuts on the principal formulas
of the rules ($\et^i_r$ $2 \leq i \leq 4$ and $\et^1_l$) or
($\ou^i_r$ $3 \leq i \leq 4$ and $\ou^1_l$) or ($\f^2_r$ and
$\f^1_l$) cannot be eliminate.
\item We can remove the rules $\et^3_r$ and $\et^4_r$ and replace
the rule $\et^1_l$ by the following rule:
\begin{center}
$\F{\G,A,B \v \D;C \; {\;}_{A \not \in {\cal
    P} \; {\rm or} \; B \not \in {\cal P}}} {\G, A \et B \v \D;C} \;
\et^1_l$
\end{center}
In this new system, the results of sections \ref{Prop} and \ref{Codings} remain
true but not those of section \ref{LL}.
\item We chose an additive ``or'' to facilitate the embedding of our
system in LL. The results of sections \ref{Prop} and \ref{Codings}
remain true if we add the two following left rules:\\

\begin{minipage}[t]{190pt}
$\F{\G , A \v \D ;C  \;\;\;  \G , B \v \D ; \; {\;}_{A \not \in {\cal
    P}}} 
{\G , A \ou B \v \D ; C} \; \ou^3_l$
\end{minipage}
\begin{minipage}[t]{190pt}
$\F{\G , A \v \D ;  \;\;\;  \G , B \v \D ; C \; {\;}_{B \not \in {\cal
    P}}} 
{\G , A \ou B \v \D ; C} \; \ou^4_l$
\end{minipage}

\item If we consider the connector $\non$ as primitive, 
we can add the following rules:\\

\begin{minipage}[t]{200pt}
$\F{\G\v\D;A}{\G, \neg A \v \D;} \; \neg^1_l$
\end{minipage}
\begin{minipage}[t]{200pt}
$\F{\G , A \v  \D;} {\G \v \D ; \neg A} \; \neg_r$\\[0.2cm]
\end{minipage}
\begin{minipage}[t]{200pt}
$\F{\G\v\D,A;\Pi}{\G,\neg A \v \D;\Pi} \; \neg^2_l$
\end{minipage}
\begin{minipage}[t]{200pt}
$\;\;\;\;\;\;\;\;\;\;\;\;\;\;\;\;\;\;\;\;\;\;\;\;\;$\\
\end{minipage}

If $\bot \in {\cal P}$, we derive $\non A \v_{\cal P} ; A \f
\bot$ and $A \f \bot \v_{\cal P} ; \non A$.
\end{enumerate}
\end{remark}

\section{Properties of system ML$_{\cal P}$} \label{Prop}

\begin{theorem} \label{cut}
The {\it Hauptsatz} holds for {\rm ML}$_{\cal P}$.
\end{theorem}

\begin{proof} 
The degree of a cut-rule in a ${\cal P}$-derivation is the pair of
integers $(l,k)$ where $l$ is the length of the cut-formula and $k$ is
defined by:

\begin{itemize}
\item $k= 3$, if it is the rule $cut_2$.
\item $k= 2$, if it is the rule $cut_1$ and the cut-formula of the
stoup of the left premise is the principal formula of a logical rule.
\item $k= 1$, if it is the rule $cut_1$ and the cut-formula of the
stoup of the left premise is the principal formula of a logical rule
but not the cut-formula of the right premise.
\item $k=0$, if it is the rule $cut_1$ and the cut-formulas are the
principal formulas of logical rules.
\end{itemize}

The order we consider on degrees is the lexicographic order.

The degree of a ${\cal P}$-derivation is the finite list of increasing
degrees of its cuts. We consider also the lexicographic order on these
degrees.

Let ${\cal D}$ be a ${\cal P}$-derivation. We will explain how to
reduce a cut in ${\cal D}$ to obtain a ${\cal P}$-derivation of
smaller degree.

We consider a cut of degree $(l,k)$ where its premises are derivable
without the cut-rules.

\begin{itemize}

\item If $k= 0$, we replace this cut by other cuts of degrees $(l',k')$
where $l'<l$.

\item If $k= 1$, we move up the left premise in the ${\cal P}$-derivation of
the right premise at the places where the cut-formula was the
principal formula of a logical rule. We thus replace this cut by other
cuts of degrees $(l,0)$.

\item If $k= 2$, we move up the right premise in the ${\cal P}$-derivation of
the left premise at the places where the cut-formula was the principal
formula of a logical rule. We thus replace this cut by other cuts of degrees
$(l,1)$ or $(l,0)$.

\item If $k = 3$, we move up the right premise in the ${\cal P}$-derivation of
the left premise at the places where the cut-formula was introduced
using the rules $der$, $w_r$ or $0$. We thus replace this rule $cut_2$ by other
rules $cut_1$ of degrees $(l,k')$ where $0 \leq k' \leq 2$.

\end{itemize}

We notice that in each case the degree of the obtained ${\cal
P}$-derivation decreases strictly.

\end{proof}

\begin{corollary}
The {\rm ML}$_{\cal P}$ has the sub-formula property.
\end{corollary}

\begin{proof}
By theorem \ref{cut}.
\end{proof}

\begin{corollary} \label{const}
\begin{enumerate}
\item If $\v_{\cal P} ; A \ou B$, then  $\v_{\cal P} ; A$ or $\v_{\cal
P} A ;$ or $\v_{\cal P} ; B$ or $\v_{\cal P} B ;$.
\item If $A,B \not \in {\cal P}$ and $\v_{\cal P} ; A \ou B$, then  $\v_{\cal P} ; A$ or $\v_{\cal P} ; B$.
\end{enumerate}
\end{corollary}

\begin{proof}
We consider a normal derivation of $\v_{\cal P} ; A \ou B$ and we look
at the last used rule.
\end{proof}

\section{Codings of LK and LJ in  ML$_{\cal P}$} \label{Codings}

We consider systems LK and LJ constructed respectively over the sets
of variables ${\cal V} \cup \{\perp\}$ and ${\cal V} \cup \{0\}$. We
suppose that $\bot$ (resp. $0$) is the symbol for the absurdity of LK
(resp. LJ). We will give some conditions to code separately systems LK
and LJ in ML$_{\cal P}$.

\begin{definition}
A set of formulas ${\cal S}$ is said to be \emph{stable} iff

for every $c \in \{\et,\ou,\f\}$, if $A c B \in {\cal S}$, then $A,B
\in {\cal S}$
\end{definition}

\begin{theorem} \label{LK}
Let ${\cal K}$ be a stable set such that ${\cal K} \subseteq {\cal P}$
and $0 \not \in {\cal K}$.

If $\G,\D \subseteq {\cal K}$, then $\G \v_{\cal P} \D;$ iff $\G
\v_{\rm LK} \D$.
\end{theorem}

\begin{proof}
$\Rightarrow$ : If we replace ``;'' by ``,'', the rules of ML$_{{\cal
P}}$ are rules of LK.

$\Leftarrow$ : We consider a derivation ${\cal D}$ of $\G \v_{\rm LK}
\D$. We check that we can move up the rules used in ${\cal D}$ without
putting formulas in the stoups. The rule $cut$ and the left rules of
LK correspond to the rules $cut_2$, $c_l$, $w_l$, $\perp$,
$\et^2_l$,$\ou^2_l$ and $\f^3_l$ (without stoups). The axiom and the
right rules of LK correspond to $ax$, $c_r$, $w_r$, $\et^2_r$,
$\ou^3_r$, $\ou^4_r$ and $\f^2_r$ using the rule $der$.
\end{proof}

\begin{remark}
\begin{enumerate}
\item The set ${\cal K}={\cal P}$ of all formulas on ${\cal V} \cup \{\perp\}$ 
satisfies the hypothesis of theorem \ref{LK}.
\item For every formula $A$ on ${\cal V} \cup \{\perp\}$, let ${\cal K}_A$ be
the finite set of the sub-formulas of $A$. If ${\cal K}_A \subseteq
{\cal P}$, then $\v_{{\cal P}} A;$ iff $\v_{\rm LK} A$.
\end{enumerate}
\end{remark}

\begin{theorem} \label{LJ}
Let ${\cal I}$ be a stable set such that ${\cal I} \cap {\cal P} =
\emptyset$ and $\perp \not \in {\cal I}$.

If $\G,A \subseteq {\cal I}$, then $\G \v_{\cal P} ; A$ iff
$\G \v_{\rm LJ} A$.
\end{theorem}

\begin{proof}
$\Rightarrow$ : We consider a normal ${\cal P}$-derivation ${\cal D}$ of $\G
\v_{\cal P} ; A$. The fact that $\G,A \subseteq {\cal I}$
allows to move up the rules using in ${\cal D}$ without putting
formulas in the bodies. Then the only rules used in ${\cal D}$ are
intuitionistic rules.

$\Leftarrow$ : The rules of LJ correspond to the rules $ax$, $cut_1$,
$c_l$, $w_l$, $0$, $\et^1_l$, $\et^1_r$, $\ou^1_l$, $\ou^1_r$,
$\ou^2_r$, $\f^1_l$ and $\f^1_r$ (without bodies).
\end{proof}

\begin{remark}
\begin{enumerate}
\item Let ${\cal P} = \{\perp\}$. The set ${\cal I}$ of all
formulas on ${\cal V} \cup \{0\}$ satisfies the hypothesis of
theorem \ref{LJ}.
\item Let $A$ be a formula on ${\cal V} \cup \{0\}$. We have
$\v_{\{\perp\}} ;A$ iff $\v_{\rm LJ} A$.
\end{enumerate}
\end{remark}

To code, at the same time, the two systems, it is necessary to realize
all the conditions of theorems \ref{LK} and \ref{LJ}. We give an
example of such a system.

\begin{definition}
\begin{enumerate}
\item We suppose that we have two disjoint sets of propositional
variables: ${\cal V}_i = \{X_i , Y_i , Z_i , ... \}$ the set of
\emph{intuitionistic variables} and ${\cal V}_c = \{X_c , Y_c , Z_c ,
... \}$ the set of \emph{classical variables}. Let ${\cal V} = {\cal
V}_i \cup {\cal V}_c$.

\item If $A$ is a formula, we denote by $var(A)$ the set of
variables and constants of $A$.

\item Let $\tilde{{\cal F}}$ be the set of all formulas,

$\tilde{{\cal K}} = \{F \in \tilde{{\cal F}}$ / $var(F) \subseteq
{\cal V}_c \cup \{\perp\}\}$ the set of \emph{classical formulas}, 

$\tilde{{\cal I}} = \{F \in \tilde{{\cal F}}$ / $var(F) \subseteq
{\cal V}_i \cup \{0\}\}$ the set of \emph{intuitionistic formulas} and

$\tilde{{\cal P}} = \tilde{{\cal F}} - \tilde{{\cal I}}$.
\end{enumerate}
\end{definition}

\begin{corollary} 
\begin{enumerate}
\item If $\G, \D \subseteq \tilde{{\cal K}}$, then $\G \v_{\tilde{{\cal P}}} \D;$ iff $\G \v_{\rm LK} \D$.
\item If $\G, A \subseteq \tilde{{\cal I}}$, then $\G \v_{\tilde{{\cal P}}} ;A$ iff $\G \v_{\rm LJ} A$.
\end{enumerate}
\end{corollary}

\begin{proof}
We use theorems \ref{LK} and \ref{LJ}.
\end{proof}

\section{Coding of ML$_{\cal P}$ in LL} \label{LL}

\begin{definition}
\begin{enumerate}
\item We define the following two \emph{translations} $b$ and $t$ from {\rm
ML}$_{\cal P}$ to {\rm LL}:
\begin{itemize}
\item if $A \not \in {\cal P}$, then $b(A) = t(A)$
\item if $A \in {\cal P}$, then $b(A) = ? t(A)$
\end{itemize}
and
\begin{itemize}
\item  $t(0) = t(\perp) = 0$
\item  $t(X) = !X$, for every $X \in {\cal V}$
\item  $t(A \et B) = !b(A) \otimes !b(B)$
\item  $t(A \ou B) = !b(A) \oplus !b(B)$
\item  $t(A \f B) = !(t(A) \fl b(B))$
\end{itemize}
\item If $\G = A_1,...,A_n$, then $t(\G) = t(A_1),...,t(A_n)$.
\end{enumerate}
\end{definition}

\begin{lemma} \label{trad}
\begin{enumerate}
\item If $\G, t(A),t(A) \v_{\rm {LL}} \D$, then $\G, t(A) \v_{\rm
{LL}} \D$.
\item If $\G \v_{\rm {LL}}\D$, then $\G,t(A) \v_{\rm {LL}} \D$.
\item If $t(\G),A \v_{\rm {LL}} ?t(\D)$, then $t(\G),? A \v_{\rm {LL}} ?t(\D)$.
\item If $t(\G) \v_{\rm {LL}} ?t(\D), A$, then $t(\G) \v_{\rm {LL}}
?t(\D), !A$.
\end{enumerate}
\end{lemma}

\begin{proof}
See \cite{DJS}.
\end{proof}

\begin{theorem} 
If $\G  \v_{\cal P} \D ; \Pi$, then $t(\G) \v_{\rm {LL}} ?t(\D),t(\Pi)$
\end{theorem}

\begin{proof}
By induction on a ${\cal P}$-derivation of $\G  \v_{\cal P} \D
; \Pi$. We look at the last rule used.
\begin{itemize}
\item For the rules $ax$, $cut_1$, $der$, $c_r$, $w_r$, $0$, $\perp$,
$\et^1_l$, $\ou^1_l$ and $\f^1_l$, the proof is easy.
\item For the rules $c_l$ and $w_l$, we use $1.$ and $2.$ of lemma \ref{trad}.
\item For the rules $cut_2$, $\et^2_l$, $\ou^2_l$ and $\f^2_l$, we use $3.$ of
lemma \ref{trad}.
\item For the rules $\et^i_r$ $(1 \leq i \leq 4)$, $\ou^i_l$ $(1 \leq
i \leq 4)$ and $\f^i_l$ $(1 \leq i \leq 2)$, we use $4.$ of
lemma \ref{trad}.
\item For the rule $\f^3_l$, we have, by induction hypothesis,
$t(\G),b(B) \v_{\rm {LL}} ?t(\D)$ and  $t(\G') \v_{\rm {LL}}p
?t(\D'),b(A),t(\Pi)$. We deduce that

$$\F{\F{\F{\F{\F{}{t(\G),b(B) \v_{\rm {LL}} ?t(\D)} \;\;\; \F{}{t(A)\v
t(A)}}{t(\G),t(A) \fl b(B),t(A) \v ?t(\D)}}{t(\G),!(t(A) \fl
b(B)),t(A)\v ?t(\D)}}{t(\G),!(t(A) \fl b(B)),b(A) \v ?t(\D)} \;\;\; \F{}{t(\G') \v_{\rm {LL}} ?t(\D'),b(A),t(\Pi)}}{t(\G),t(\G'),t(A \f B) \v
?t(\D), ?t(\D'), t(\Pi)}$$

\end{itemize}
\end{proof}

\end{document}